\documentclass[12pt,leqno]{amsart}
\usepackage{amsmath,amssymb,amsfonts}
\usepackage[all]{xy}

\usepackage[T5,T1]{fontenc}

\usepackage{mathpazo}
\usepackage{hyperref}
\usepackage{a4wide}
\numberwithin{equation}{section}

\theoremstyle{plain}
\newtheorem{thm}{Theorem}[section]
\newtheorem{lem}[thm]{Lemma}
\newtheorem{cor}[thm]{Corollary}
\newtheorem{prop}[thm]{Proposition}

\theoremstyle{definition}
\newtheorem{defn}[thm]{Definition}

\newcommand{\im}{{\rm im}}

\newcommand{\sC}{{\mathcal C}}

\newcommand{\sS}{{\mathcal S}}
\newcommand{\sT}{{\mathcal T}}

\newcommand{\F}{{\mathbb F}}

\newcommand{\Q}{{\mathbb Q}}

\newcommand{\U}{{\mathbb U}}

\newcommand{\Z}{{\mathbb Z}}

\def\NDT{{\fontencoding{T5}\selectfont Nguy\~ \ecircumflex n Duy T\^an}}
 
\begin{document}
\title{Description of Galois unipotent extensions} 

 \author{ Masoud Ataei, J\'an Min\'a\v{c} and \NDT}
\address{Department of Mathematics, Western University, London, Ontario, Canada N6A 5B7}
\email{mataeija@uwo.ca}
\address{Department of Mathematics, Western University, London, Ontario, Canada N6A 5B7}
\email{minac@uwo.ca}
 \address{Department of Mathematics, Western University, London, Ontario, Canada N6A 5B7 and Institute of Mathematics, Vietnam Academy of Science and Technology, 18 Hoang Quoc Viet, 10307, Hanoi - Vietnam } 
\email{duytan@math.ac.vn}
\thanks{JM is partially supported  by the Natural Sciences and Engineering Research Council of Canada (NSERC) grant R0370A01. NDT is partially supported  by the National Foundation for Science and Technology Development (NAFOSTED) grant 101.04-2014.34}
 \begin{abstract}
Given an arbitrary field $F$, we describe all Galois extensions $L/F$ whose \mbox{Galois} groups are isomorphic to the group of upper triangular unipotent 4-by-4  matrices with entries in the field of two elements.
\end{abstract}
\dedicatory{Dedicated to Professor Paulo Ribenboim} 
\maketitle
\section{Introduction}

Let $G$ be a finite group, and let $F$ be an arbitrary field. A fundamental problem in Galois theory is to describe all Galois extensions $L/F$ whose Galois groups are isomorphic to group $G$. It is desirable to describe such families of extensions using invariants of $L/F$ which depend only on the base field $F$. If $G$ is abelian then this is possible by the theories of Kummer and Artin-Schreier's extension, and classical work  of A. Allbert and D. J. Saltman. Moreover this description is elegant, simple and useful. It is known that there are  some other very interesting and useful explicit constructions of Galois extensions $L/F$ with prescribed Galois group $G$. See for example, \cite{Ja}, \cite[Chapters 5-6]{JLY}, \cite[Chapters 2,5-7]{Le}, \cite{Ma}, \cite{MNg}, \cite{MZ}, \cite{Sa}. However the simplicity and generality of the descriptions of Kummer and Artin-Schreier's extension seem to be unmatched.

Recall that for each natural number $n$, $\U_n(\F_p)$ is the group of  upper triangular $n\times n$-matrices with entries in $\F_p$ and diagonal entries 1. In a recent development of  Massey products in Galois cohomology, it was recognized that Galois extensions $L/F$ with ${\rm Gal}(L/F)\simeq \U_n(\F_p)$ play a very special role in Galois theory of $p$-extensions. (See \cite{Ef}, \cite{EMa}, \cite{HW}, \cite{Dwy}, \cite{GLMS}, \cite{MT1, MT2, MT4, MT5}.) Moreover the works above reveal some surprising  depth and simplicity of analysis of these extensions. 
The main purpose of our paper  is to describe all  Galois extensions $L/F$ with ${\rm Gal}(L/F)\simeq \U_4(\F_2)$ over any given field $F$. Our main results are Theorem~\ref{thm:U4 char not 2} and Theorem~\ref{thm:U4 char 2}.
We also show that a similar description is valid for Galois extensions with Galois group isomorphic to $\U_3(\F_2)$ over an arbitrary field. (Note that $\U_3(\F_2)$ is isomorphic to the dihedral group of order 8.)

Beside of their intrinsic value, these simple descriptions of Galois extensions $L/F$ with ${\rm Gal}(L/F)\simeq \U_4(\F_2)$ are expected to play a significant role in an induction approach to the construction of Galois extensions $L/F$ with ${\rm Gal}(L/F)\simeq \U_n(\F_2)$ for $n\geq 2$, and for a possible proof of the Vanishing $n$-Massey Conjecture for absolute Galois groups of fields. (See \cite {MT1,MT5}.) Also this description should be useful for establishing the Kernel $n$-Unipotent Conjecture for absolute Galois groups of fields and $p=2$. This would be a very interesting extension of the work of \cite{MSp}, \cite{Vi}. (See also \cite{EM},\cite{MT2}.) 

Indeed a natural program for solving the Vanishing $n$-Massey Conjecture for absolute Galois groups of fields uses Theorem 2.4 in \cite{Dwy}. This theorem reduces the problem to solving a certain Galois embedding problem induced from the central extension
\[
1\to \F_p \to \U_{n+1}(\F_p) \to \bar{\U}_{n+1}(\F_p)\to 1,
\]
where $\bar{\U}_{n+1}(\F_p)$ is the quotient of $\U_{n+1}(\F_p)$ by its center. 
The most interesting and difficult task is in fact to find a construction of Galois extensions $L/F$ with ${\rm Gal}(L/F)\simeq \U_{n+1}(\F_p)$ that solve the embedding problem. Because $\U_{n+1}(\F_p)$ contains copies of $\U_n(\F_p)$, one can consider to use induction on $n$. This program was realized in \cite{MT5} in the case $n=3$. Here the knowledge of the explicit construction of Galois extensions with Galois group $\U_3(\F_p)$ was crucial. A successful implementation of this program also in the case $n=4$ may reveal the induction procedure which is valid in general. Therefore the knowledge of Galois extensions $L/F$ with ${\rm Gal}(L/F)\simeq \U_4(\F_2)$ seems to be important for the implementation of this program. 

  Further possible applications of this work can be related to an extension of the study of Redei symbols and also the
study of 2-Hilbert towers.  (See \cite{A}, \cite{McL}.) 

Next we shall briefly describe  the content of our paper. In Section 2 we provide  a description of Galois extensions with Galois group isomorphic to $\U_4(\F_2)$ over a given field of characteristic not 2. We then use this description to count the number of Galois extensions  with Galois group isomorphic to $\U_4(\F_2)$ over a field which is a finite extension of $\Q_2$. In Section 3 we provide a description of Galois dihedral extensions of order $8$ over a given field of characteristic not 2. In Section 4 we provide  a description of Galois extensions  with Galois group isomorphic to $\U_4(\F_2)$ over a given field of characteristic 2. We then use this description to count the number of Galois extensions  with Galois group isomorphic to $\U_4(\F_2)$ over a field $F$ with $F/\wp(F)$ finite, where $\wp(X)=X^2-X$ is the Artin-Schreier polynomial. (Here $F/\wp(F)$ is the quotient group of $F$ by its subgroup $\wp(F)$ of all values of $\wp$.) Finally in Section 5 we illustrate our results by an example with base field $\Q_2$. Here we provide a list of all unipotent Galois extensions $L/\Q_2$ with Galois groups isomorphic to $\U_n(\F_2)$ for $n\geq 2$. This completes the work of Naito (\cite{Na}) who listed all dihedral extensions of order 8 over $\Q_2$.
\\
\\
{\bf Acknowledgements: } We are grateful to I. Efrat,  S. Gille, M. Hopkins, E. Matzri, S. Sorkhou, A. Topaz and K. Wickelgren for interesting discussions concerning previous work on Massey products in Galois cohomology which was among the inspiration for this work although it is strictly speaking logically independent from these considerations. We are  grateful to   anonymous referees for their  careful reading of our paper and for providing us with insightful comments and valuable suggestions which we used to improve our exposition.
\\
\\
{\bf Notation:} For any field $F$ of characteristic not 2 and for any element $a\in F$, we denote $[a]_F$  the image of $a$ in $F^\times/(F^\times)^2$. For $V$ an $\F_2$-subspace of $F^\times/(F^{\times})^2$, we define $F(\sqrt{V})=F(\sqrt{v}\colon [v]_F\in V)$. For $a, b$ in  $F$, $(a, b)_F$ or simply $(a,b)$  is the corresponding  quaternion algebra. (See \cite[Chapter 3]{Lam}.) We write $(a,b)=0$ if this algebra is isomorphic to the matrix algebra of $2\times 2$-matrices over $F$. 

For any field $F$ of characteristic 2 and for any element $a\in F$, we denote $[a]_F$  the image of $a$ in $F/\wp(F)$. 

For a finite field extension $E/F$, we use ${\rm Nm}_{E/F}$ and ${\rm Tr}_{E/F}$ to denote the norm and trace maps respectively. If $E/F$ is Galois with Galois group isomorphic to a finite group $G$, we say that $E/F$ is a $G$-extension.

For $1 \leq i, j \leq n$, let $e_{ij}$ denote the $n$-by-$n$ matrix with the 1 of $\F_p$ in the position $(i, j)$ and 0 elsewhere, and let $E_{ij}=1+e_{ij}$.

We denote $D_8$ the dihedral group of order 8.
\\
\\
{\bf Convention:} For a given base field $F$, all extensions over $F$ considered in this paper are inside a chosen separable closure of $F$.
\section{Description of Galois $\U_4(\F_2)$-extensions: The case of characteristic not 2}

Let $F$ be a field of characteristic different from 2. 
\begin{defn}
A pair $([b]_F, V)$, where $b$ is in $F^\times$ and $V\subseteq F^\times/(F^\times)^2$, is {\it admissible} if $\dim_{\F_2}(V)=2$, $\dim_{\F_2}(\langle V,[b]_F\rangle)=3$ and $(b,v)=0$ for every $[v]_F\in V$.
\end{defn}

\begin{lem} Assume that $([b]_F,V)$ is admissible. Let $E=F(\sqrt{V})$.  Then there exists $\delta\in E$ such that $[{\rm Nm}_{E/F}(\delta)]_F=[b]_F$. 
\end{lem}
\begin{proof} We have $V=\langle [a]_F, [c]_F\rangle$ for some $a,c\in F^\times$. Then $(a,b)=(b,c)=0$. By \cite[Section 5]{MT1},  there exists $\delta\in E$ such that
${\rm Nm}_{E/F}(\delta)=bd^2$ for some $d\in F^\times$.
\end{proof}
\begin{defn} 
Assume that $([b]_F,V)$ is admissible. Let $E=F(\sqrt{V})$. Then a triple $([b]_F, V, W)$, where $W$ is a free  $\F_2[{\rm Gal}(E/F)]$-submodule of $E^\times/(E^\times)^2$, is {\it admissible} if $W$ is generated by an element $[\delta]_E$ with $[{\rm Nm}_{E/F}(\delta)]_F =[b]_F$.
\end{defn}

\begin{lem}
\label{lem:modular group ring}
 Let $K$ be a field of characteristic  $p>0$. Let $G$ be a finite $p$-group. Then every non-zero left ideal in the group ring $K[G]$ contains the element $\sum_{\sigma\in G} \sigma$.
\end{lem}
\begin{proof} Let $I$ be any non-zero left ideal in $K[G]$. Then $I$ contains a minimal non-zero left ideal $J$. As a $K[G]$-module, $J$ is simple. We know that over $K[G]$ there is up to isomorphism only one simple module, which is $K$ with trivial action. Let $n$ be any element in $J$ which generates $J$ as a $K[G]$-module. Then $n$ is fixed under all elements of $G$. Hence $n= a\sum_{\sigma\in G}\sigma$, for some $a\in K^\times$. This implies that $\sum_{\sigma\in G}\sigma$ is in $J$. 
\end{proof}
\begin{lem} 
\label{lem:generator}
Let $([b]_F, V, W)$ be an admissible triple. Assume that $V=\langle [a]_F,[c]_F\rangle$. Let $E=F(\sqrt{V})$. Assume that $W$ is generated by $[\delta]_E$ as a free $\F_2[{\rm Gal}(E/F)]$-module with $[{\rm Nm}_{E/F}(\delta)]_F =[b]_F$. Let $A= {\rm Nm}_{E/F(\sqrt{a})}(\delta)$ and $C={\rm Nm}_{E/F(\sqrt{c})}(\delta)$. Then every generator of $W$ as a free $\F_2[{\rm Gal}(E/F)]$-module is of the form 
\[[\delta^\prime]_E= [\delta A^{\epsilon_A}C^{\epsilon_C} b^{\epsilon_b}]_E,\]
where $\epsilon_A,\epsilon_C,\epsilon_b\in \{0,1\}$.

Furthermore for any generator $[\delta^\prime]_E$ of $W$ as a free $\F_2[{\rm Gal}(E/F)]$-module, we have $[{\rm Nm}_{E/F}(\delta^\prime)]_F=[b]_F$. In particular, this implies that the pair $(V,W)$ uniquely determines  $[b]_F$. 
\end{lem}
\begin{proof} Let $G={\rm Gal}(E/F)$. 
As an $\F_2$-vector space, $W$ is generated  by $[\delta]_E, [A]_E, [C]_E, [b]_E$.
Let $[\delta^\prime]_E$ be an arbitrary generator of the free $\F_2[G]$-module. Then 
\[[\delta^\prime]_E= [\delta^{\epsilon_\delta} A^{\epsilon_A}C^{\epsilon_C} b^{\epsilon_b}]_E,\]
for some $\epsilon_{\delta}, \epsilon_A,\epsilon_C,\epsilon_b\in \{0,1\}$.
Suppose that $\epsilon_{\delta}=0$, then we see that $(\sum_{\sigma\in G}\sigma)([\delta^\prime]_E)$ is trivial in $E^\times/((E^\times)^2)$, a contradiction. Hence $\epsilon_{\delta}=1$. Furthermore, we have
\[
[{\rm Nm}_{E/F}(\delta^\prime)]_F= [b]_F.
\]
This implies that $[b]_F$ is uniquely determined by $V$ and $W$.

Conversely, assume that  $[\delta^\prime]_E=[\delta A^{\epsilon_A}C^{\epsilon_C} b^{\epsilon_b}]_E,$ for some $\epsilon_A,\epsilon_C,\epsilon_b\in \{0,1\}$. Let $W^\prime$ be the $\F_2[G]$-module generated by $[\delta^\prime]_E$. Then we have $W^\prime\subseteq W$. It is then enough to show that $W^\prime$ is a free $\F_2[G]$-module. Suppose that $W^\prime$ would not be free. Then there would exist a non-zero ideal $I\subseteq \F_2[G]$ such that $I$ would annihilate $\delta^\prime$. By Lemma~\ref{lem:modular group ring} any non-zero ideal of $\F_2[G]$ contains the element $\sum_{\sigma\in G}\sigma=:N$. Therefore $N$ would annihilate $[\delta^\prime]_E$. This contradicts  the fact that
\[
N([\delta^\prime]_E)]=[{\rm Nm}_{E/F}(\delta^\prime)]_E= [b]_E \not=1\in E^\times/(E^\times)^2.
\qedhere
\] 
\end{proof}

\begin{prop}
\label{prop:admissible triple}
 Let  $([b]_F, V, W)$ be an admissible triple. Let $E=F(\sqrt{V})$. Let $L=E(\sqrt{W})$. Then  $L/F$ is a Galois $\U_4(\F_2)$-extension.
\end{prop}
\begin{proof}
Suppose that $V=\langle [a]_F,[c]_F\rangle$ and that $W$ is generated by $\delta$ with ${\rm Nm}_{E/F}(\delta)=bd^2$. Let $A= {\rm Nm}_{E/F(\sqrt{a})}(\delta)$ and $C={\rm Nm}_{E/F(\sqrt{c})}(\delta)$. 
We first note that $F(\sqrt{a},\sqrt{b},\sqrt{c})/F$ is an abelian 2-elementary extension whose Galois group is generated by $\sigma_a,\sigma_b,\sigma_c$, where
\[
\begin{aligned}
\sigma_a(\sqrt{a})&=- \sqrt{a}, \sigma_a(\sqrt{b})=\sqrt{b}, \sigma_a(\sqrt{c})=\sqrt{c};\\
\sigma_b(\sqrt{a})&=\sqrt{a}, \sigma_b(\sqrt{b})= -\sqrt{b}, \sigma_b(\sqrt{c})= \sqrt{c};\\
\sigma_c(\sqrt{a})&=\sqrt{a}, \sigma_c(\sqrt{b})=\sqrt{b}, \sigma_c(\sqrt{c})=- \sqrt{c}.
\end{aligned}
\]

Clearly we have
\[
\begin{aligned}
\sigma_c(\delta)&=\delta A \delta^{-2},\\
\sigma_a(\delta)&=\delta C \delta^{-2}, \\
\sigma_a(A) &= A \frac{bd^2}{A^2},\\
\sigma_c(C) &= C\frac{bd^2}{C^2},\\
\end{aligned}
\]
and 
\[
\frac{C}{A}= \frac{\sigma_a(\delta)}{\delta} \frac{\delta}{\sigma_c(\delta)}.
\]
Then \cite[Section 3]{MT3} implies that $L/F$ is a Galois $\U_4(\F_2)$-extension.  Moreover an explicit isomorphism $\rho\colon {\rm Gal}(L/F)\to \U_4(\F_2)$ is given by
\[
\tilde{\sigma}_a \mapsto E_{12}, \; \;
\tilde{\sigma}_b\mapsto  E_{23}, \;\;
\tilde{\sigma}_c\mapsto E_{34},
\]
for suitable extensions  $\tilde{\sigma}_a,\tilde{\sigma}_b,\tilde{\sigma}_c\in {\rm Gal}(L/F)$ of $\sigma_a,\sigma_b,\sigma_c$.
\end{proof}
\begin{prop}
\label{prop:Galois extension}
There is a natural way to associate an admissible triple $([b]_F,V,W)$ to any given Galois $\U_4(\F_2)$-extension $L/F$.
\end{prop}
\begin{proof}
Assume that $L/F$ is a Galois $\U_4(\F_2)$-extension. Let $\rho\colon {\rm Gal}(L/F)\to \U_4(\F_2)$ be any isomorphism. Set $\sigma_1=\rho^{-1}(E_{12})$, $\sigma_2=\rho^{-1}(E_{23})$, and $\sigma_3=\rho^{-1}(E_{34})$. Then the commutator subgroup $\Phi=[{\rm Gal}(L/F),{\rm Gal}(L/F)]$ is the internal direct sum 
\[
\Phi=\langle [\sigma_1,\sigma_2]\rangle \oplus \langle [\sigma_2,\sigma_3]\rangle \oplus \langle [[\sigma_1,\sigma_2],\sigma_3]\rangle \simeq (\Z/2\Z)^3.
\]

 Let $M$ be the fixed field of $\Phi$. Then $M/F$ is an  abelian 2-elementary extension of $F$, and ${\rm Gal}(M/F)$ is the internal direct sum 
\[
{\rm Gal}(M/F)=\langle \sigma_1|_M \rangle\oplus \langle \sigma_2|_M\rangle \oplus \langle \sigma_3|_M\rangle \simeq (\Z/2\Z)^3.
\]
Let $[a]_F,[b]_F,[c]_F$ be elements in $F^\times/(F^\times)^2$ which is dual to $\sigma_1|_M, \sigma_2|_M, \sigma_3|_M$ respectively via  Kummer theory. Explicitly we require that
\[
\begin{aligned}
\sigma_1(\sqrt{a})=-\sqrt{a},\sigma_1(\sqrt{b})=\sqrt{b}, \sigma_1(\sqrt{c})=\sqrt{c};\\
\sigma_2(\sqrt{a})=\sqrt{a},\sigma_2(\sqrt{b})=-\sqrt{b}, \sigma_2(\sqrt{c})=\sqrt{c};\\
\sigma_3(\sqrt{a})=\sqrt{a},\sigma_3(\sqrt{b})=\sqrt{b}, \sigma_3(\sqrt{c})=-\sqrt{c}. 
\end{aligned}
\]
Let $E=F(\sqrt{a},\sqrt{c})$. Then $E$ is fixed under $\sigma_2$, $[\sigma_1,\sigma_2]$, $[\sigma_2,\sigma_3]$ and $[[\sigma_1,\sigma_2],\sigma_3]$. 
 Hence $E$ is fixed under a subgroup $H$ of ${\rm Gal(L/F)}$ which is generated by $\sigma_2$, $[\sigma_1,\sigma_2]$, $[\sigma_2,\sigma_3]$ and $[[\sigma_1,\sigma_2],\sigma_3]$. We have $[L^H:F]=|{\rm Gal}(L/F)|/|H|=4$,  and $[E:F]=4$. Therefore $E=L^H$.
\\
\\ 
{\bf Claim:} $E$ does not depend on the choice of $\rho$.  

\noindent
{\it Proof of Claim:} Suppose that $\rho^\prime\colon {\rm Gal}(L/F)\to \U_4(\F_2)$ is another isomorphism. We define $\sigma^\prime_1={\rho^\prime}^{-1}(E_{12})$, $\sigma^\prime_2={\rho^\prime}^{-1}(E_{23})$, and $\sigma^\prime_3={\rho^\prime}^{-1}(E_{34})$. Let $H^\prime$ be the group generated by $\sigma_2^\prime$, $[\sigma_1^\prime,\sigma_2^\prime]$, $[\sigma_2^\prime,\sigma_3^\prime]$ and $[[\sigma_1^\prime,\sigma_2^\prime],\sigma_3^\prime]$. We need to show that $H=H^\prime$. We first note that $\sigma_2$ and $\sigma_2^\prime$ commute with every element in $\Phi$.

Clearly $\sigma_2^\prime|_M$ is in ${\rm Gal}(M/F)=\langle \sigma_1|_M \rangle\oplus \langle \sigma_2|_M\rangle \oplus \langle \sigma_3|_M\rangle$. 

Hence modulo the subgroup $\Phi$, $\sigma_2^\prime$ is equal to one of the following elements $\sigma_1,\sigma_2,\sigma_3$, $\sigma_1\sigma_2,\sigma_1\sigma_3,\sigma_2\sigma_3$, $\sigma_1\sigma_2\sigma_3$.

If $\sigma_2^\prime=\sigma_1,$ or $\sigma_1\sigma_2$, or $\sigma_1\sigma_3$, or $\sigma_1\sigma_2\sigma_3$ modulo $\Phi$, then 
\[
[[\sigma_2,\sigma_3],\sigma_2^\prime]=[[\sigma_2,\sigma_3],\sigma_1],
\] 
which is impossible since $[[\sigma_2,\sigma_3],\sigma_1]$ is nontrivial but $[[\sigma_2,\sigma_3],\sigma^\prime_2]$ is trivial.

If $\sigma_2^\prime=\sigma_3,$ or $\sigma_2\sigma_3$ modulo $\Phi$, then 
\[
[[\sigma_1,\sigma_2],\sigma_2^\prime]=[[\sigma_1,\sigma_2],\sigma_3],
\] 
which is impossible since $[[\sigma_1,\sigma_2],\sigma_3]$ is nontrivial but $[[\sigma_1,\sigma_2],\sigma^\prime_2]$ is trivial.

From the above discussion we see that $\sigma^\prime_2\equiv \sigma_2\bmod \Phi$. This implies that   $H^\prime=H$. Thus $E$ does not depend on the choice of $\rho$.
\\
\\  
We have an exact sequence 
\[
1 \to {\rm Gal}(L/E)  \to {\rm Gal}(L/F) \to {\rm Gal}(E/F)=G \to 1.
\]
Then ${\rm Gal}(L/E)$ is an $\F_2[G]$-module where the action is by conjugation. We also have the $G$-equivariant Kummer pairing  (\cite[Section 1]{Wa2})
\[
\frac{E\cap (L^\times)^2}{(E^\times)^2}\times {\rm Gal}(L/E) \to \F_2.
\]
As an $\F_2$-vector space, ${\rm Gal}(L/E)$ has a basis consisting of $\sigma_2$, $[\sigma_1,\sigma_2]$, $[\sigma_2,\sigma_3]$ and $[[\sigma_1,\sigma_2],\sigma_3]$. 
Let $[\delta]_E$ be an element dual to $[[\sigma_1,\sigma_2],\sigma_3]$. 
Then ${\rm Nm}_{E/F}(\delta)\equiv b \bmod (E^{\times})^2$. Hence ${\rm Nm}_{E/F}(\delta)$ is in $b(F^{\times})^2 \cup ba(F^{\times})^2 \cup bc(F^{\times})^2\cup bac(F^{\times})^2$. 

Let $A={\rm Nm}_{E/F(\sqrt{a})}(\delta)$ and $C={\rm Nm}_{E/F(\sqrt{c})}(\delta)$. 
Suppose that ${\rm Nm}_{E/F}(\delta)\equiv ba \mod (F^\times)^2$. Then ${\rm Nm}_{F(\sqrt{a})/F}(A)= ba f^2$ for some $f\in F^\times$. From $\sigma_1(A)/A = ba (f/A)^2$, we see that
\[
\sigma_1(\sqrt{A})=(\pm) \sqrt{A} \sqrt{ba} f/A. 
\]
Hence
\[
\begin{aligned}
\sigma_1^2(\sqrt{A}) &= (\pm)\sigma_1(\sqrt{A}) \sigma_1(\sqrt{ba}) (f/\sigma_1(A))\\
&=(\pm)^2 \sqrt{A} \sqrt{ba} (f/A) \sqrt{b} (-\sqrt{a}) (f/\sigma_1(A))\\
&= -\sqrt{A}.
\end{aligned}
\]
This implies that $\sigma_1$ is not of order 2, a contradiction. Hence ${\rm Nm}_{E/F}(\delta)$ is not in  $ba  (F^\times)^2$.

Similarly we can show that ${\rm Nm}_{E/F}(\delta)$ is not in $bc (F^\times)^2 \cup ba(F^\times)^2$. Therefore 
\[
{\rm Nm}_{E/F}(\delta)\equiv b \bmod (F^\times)^2.
\]

 We set $V=\langle [a]_F,[c]_F\rangle$. Then $V$ does not depend on the choice of $\rho$. Since 
\[ {\rm Nm}_{F(\sqrt{a})/F}(A)={\rm Nm}_{E/F}(\delta)=b \bmod (F^\times)^2,\]
 we have $(a,b)=0$. Similarly, we have $(b,c)=0$. Therefore $(b,v)=0$ for every $v\in V$, and the pair $([b]_F,V)$ is admissible.
Let $W$ be the $\F_2[G]$-submodule of $E^\times/(E^\times)^2$ which is dual via  Kummer theory to ${\rm Gal}(L/E)$. Then $W$ does not depend on the choice of $\rho$, and $W$ is free and generated by $\delta$.  
Since $[{\rm Nm}_{E/F}(\delta)]_F=[b]_F$, we see that the triple $([b]_F,V,W)$ is admissible. Since $V$ and $W$ determine $[b]_F$ uniquely, we see that $[b]_F$ does not depend on the choice of $\rho$.
\end{proof}

\begin{thm}
\label{thm:U4 char not 2}
  Let $F$  be a field of characteristic not 2. There is a natural one-to-one correspondence between the set of admissible triples $([b]_F,V,W)$ and the set of Galois $\U_4(\F_2)$-extensions $L/F$. 
\end{thm}
\begin{proof} 
By Proposition~\ref{prop:admissible triple} we have a map $\mu$ from the set of admissible triples  $([b]_F,V,W)$ to the set of Galois $\U_4(\F_2)$-extensions $L/F$.
By Proposition~\ref{prop:Galois extension} we have a map $\eta$ from the set of Galois $\U_4(\F_2)$-extensions $L/F$ to  the set of admissible triple  $([b]_F,V,W)$. We show that $\mu$ and $\eta$ are the inverses of each other.

Let $([b]_F,V,W)$ be an admissible triple. Via the map $\mu$ we obtain a $\U_4(\F_2)$-extension $L/F$. Explicitly, if $V=\langle [a]_F,[c]_F\rangle$ and $E=F(\sqrt{a},\sqrt{c})$, then $L=E(\sqrt{W})$ and there is an isomorphism $\rho\colon {\rm Gal}(L/F)\simeq \U_4(\F_2)$ such that
$\rho^{-1}(E_{12})=\sigma_a$, $\rho^{-1}(E_{23})=\sigma_{b}$, $\rho^{-1}(E_{34})=\sigma_c$. (Here $\sigma_a,\sigma_b,\sigma_c$ are defined as in Proposition~\ref{prop:admissible triple}.)
We apply the construction in Proposition~\ref{prop:Galois extension} with this isomorphism $\rho$. Then we obtain back the admissible triple $([b]_F,V,W)$.

Now let $L/F$ be a $\U_4(\F_2)$-extension. Then via the map $\eta$ we obtain an admissible triple $([b]_F,V,W)$. Since $L=F(\sqrt{V})(W)$, we see that $\mu$ sends the triple $([b]_F,V,W)$ back to the extension $L/F$. 
\end{proof}

We apply the theorem above to count the number of Galois $\U_4(\F_2)$-extensions over a 2-adic  field. 
\begin{lem} 
\label{lem:1}
Assume that $F$ is a finite extension of $\Q_2$ of degree $n$. 
Let $q$ be the highest power of $2$ such that $F$ contains a primitive $q$-th root of unity. Then the number $N$ of admissible pairs $([b]_F,V)$ is 
\[
\begin{cases} 
\displaystyle
\frac{4(2^{n+2}-1)(2^n-1)(2^{n-1}-1)}{3} &\text{ if $q\not=2$},\\
\displaystyle \frac{4(2^{n+1}-1)(2^n-1)^2}{3} &\text{ if $q=2$}.
\end{cases}
\]
\end{lem}
\begin{proof}
Let $N^\prime$ be the number of $([a]_F,[b]_F,[c]_F)$ such that $(a,b)=(b,c)=0$ and that $\dim_{\F_2} \langle [a]_F,[b]_F,[c]_F\rangle=3$. Then $N=N^\prime/6$. This is  because for each given $V$ such that $([b]_F,V)$ is admissible, there are precisely 6 choices of choosing $([a]_F,[c]_F)$ with $V=\langle [a]_F,[c]_F\rangle$. On the other hand, by \cite[Lemma 3.6 and Proposition 3.4]{MT2}, we have
\[
N^\prime=
\begin{cases} 
\displaystyle
(2^{n+2}-1)(2^{n+1}-2)(2^{n+1}-4) &\text{ if $q\not=2$},\\
\displaystyle (2^{n+1}-1)(2^{n+1}-2)(2^{n+2}-4) &\text{ if $q=2$}.
\end{cases}
\]
The result then follows.
\end{proof}

\begin{lem}
\label{lem:2}
 Assume that $F$ is a finite extension of $\Q_2$ of degree $n$. Let us fix an admissible pair $([b]_F,V)$. Then the number of admissible triples $([b]_F,V,W)$ is $2^{3n-1}$. 
\end{lem}
\begin{proof} Let $E=F(\sqrt{V})$.
By local class field theory we have an isomorphism
\[
\dfrac{F^\times}{{\rm Nm}_{E/F}(E^\times)} \simeq {\rm Gal}(E/F)=G.
\]
Since $G$ is of exponent 2, we see that $\dfrac{F^\times}{{\rm Nm}_{E/F}(E^\times)}$ is also of exponent 2. Hence 
\[ (F^\times)^2\subseteq {\rm Nm}_{E/F}(E^\times) \subseteq F^\times.\] 
Since $|G|=4$, we have 
\[
 4=\left| \dfrac{F^\times}{{\rm Nm}_{E/F}(E^\times)}\right|
= \left[\dfrac{F^\times}{(F^\times)^2} : \dfrac{{\rm Nm}_{E/F}(E^\times)}{(F^\times)^2}\right].
\]
By \cite[Chapter II, \S 5, Corollary 5.8]{Neu}, one has $|F^\times/(F^\times)^2|=2^{n+2}$. 
Hence 
\[\left| \dfrac{{\rm Nm}_{E/F}(E^\times)}{(F^\times)^2}\right|=\left|\dfrac{F^\times}{(F^\times)^2}\right|/4=2^{n}.\]

Consider the homomorphism ${\rm Nm}\colon \dfrac{E^\times}{(E^\times)^2}\to \dfrac{F^\times}{(F^\times)^2}$. Then $\im({\rm Nm})= \dfrac{{\rm Nm}_{E/F}(E^\times)}{(F^\times)^2}$.
By \cite[Chapter II, \S 5, Corollary 5.8]{Neu}, one has $|E^\times/(E^\times)^2|=2^{4n+2}$. Hence we have
\[
|\ker {\rm Nm}|= \left|\dfrac{E^\times}{(E^\times)^2}\right|/ |\im({\rm Nm})|= 2^{4n+2}/2^n=2^{3n+2}.
\]
Hence
\[
| \{ [\delta]_E \colon [{\rm Nm}_{E/F}(\delta)]_F =[b]_F\}|= |\ker {\rm Nm}|=2^{3n+2}. 
\]
Therefore by Lemma~\ref{lem:generator} the number of $W$ such that $([b]_F,V,W)$ is admissible, is $2^{3n+2}/8=2^{3n-1}$.
\end{proof}

We recover the following result, which was also obtained in \cite[Theorem 3.8]{MT2}.
\begin{cor} Assume that $F$ is a finite extension of $\Q_2$ of degree $n$. Let $q$ be the highest power of $2$ such that $F$ contains a
primitive $q$-th root of unity. Then the number of Galois $\U_4(\F_2)$-extensions of $F$ is 
\[
\begin{cases} 
\displaystyle
\frac{(2^{n+2}-1)(2^n-1)(2^{n-1}-1)2^{3n+1}}{3} &\text{ if $q\not=2$},\\
\displaystyle \frac{(2^{n+1}-1)(2^n-1)^22^{3n+1}}{3} &\text{ if $q=2$}.
\end{cases}
\]
\end{cor}
\begin{proof} This follows from  Theorem~\ref{thm:U4 char not 2}, Lemma~\ref{lem:1} and Lemma~\ref{lem:2}. 
\end{proof}
\section{Description of Galois $D_8$-extensions}
\subsection{The case of characteristic not 2} 

Let $F$ be a field of characteristic not 2. 
\begin{defn}
An unordered pair $\{[b]_F, [a]_F\}$, where $a$ and $b$ are in $F^\times$, is {\it admissible} if  $(b,a)=0$ and $\dim_{\F_2}(\langle [a]_F,[b]_F\rangle)=2$.
\end{defn}

\begin{lem} Assume that $\{[b]_F,[a]_F\}$ is admissible. Let $E=F(\sqrt{a},\sqrt{b})$.  Then there exists $\delta_1\in F(\sqrt{a})$ such that 
\[
[{\rm Nm}_{F(\sqrt{a})/F}(\delta_1)]_F=[b]_F.
\]
Furthermore for any such $\delta_1$, there exists $\delta_2$ in $F(\sqrt{b})$ such that ${[\delta_1]_E} =[\delta_2]_E$ and 
\[
[{\rm Nm}_{F(\sqrt{b})/F}(\delta_2)]_F=[a]_F.
\]
\end{lem}
\begin{proof} 
As $(a,b)=0$ there exists $\delta_1\in F(\sqrt{a})$ (\cite[Chapter XIV, Proposition 4]{Se}) such that 
\[
[{\rm Nm}_{F(\sqrt{a})/F}(\delta_1)]_F=[b]_F.
\]
Now let $\delta$ be any element in $F(\sqrt{a})$ such that $[{\rm Nm}_{F(\sqrt{a})/F}(\delta)]_F=[b]_F$. We write $\delta= x+ y\sqrt{a}$, where $x,y\in F^\times$. Then $x^2=y^2a+bd^2$, for some $d\in F^\times$. Hence
\[
(x+y\sqrt{a}+d\sqrt{b})^2= 2(x+y\sqrt{a})(x+d\sqrt{b}).
\]
Set $\delta_2=2(x+d\sqrt{b})\in F(\sqrt{b})$. Then $[\delta_1]_E=[\delta_2]_E$ and $[{\rm Nm}_{F(\sqrt{b})/F}(\delta_2)]_F= [4(x^2-bd^2)]_F=[4y^2a]_F=[a]_F$. 
\end{proof}

The above lemma shows that the following definition is well-defined.
\begin{defn} Let $P=\{[b]_F, [a]_F\}$ be an admissible unordered pair. Let $E=F(\sqrt{a},\sqrt{b})$.

A one dimensional $\F_2$-subspace $W$ of $E^\times/(E^\times)^2$ is said to be {\it compatible} with $P$ if $W$ is generated by a $\delta\in F(\sqrt{a})$ with $[{\rm Nm}_{F(\sqrt{a})/F}(\delta)]_F=[b]_F$.
In this case we say that $(P,W)$ is {\it admissible}.
\end{defn}

The construction of Galois $D_8$-extension over fields of characteristic not 2 is known. See for example \cite[Theorem 2.2.7]{JLY}. Here we make a description of all Galois $D_8$-extensions over a given field, which is similar to the description of Galois $\U_4(\F_2)$ extensions in Theorem~\ref{thm:U4 char not 2}.

\begin{thm}
\label{thm:D4 char not 2}
  Let $F$  be a field of characteristic not 2.
There is a natural one-to-one correspondence between the set of admissible pairs $(\{[a]_F,[b]_F\},W)$ and the set of Galois $D_8$-extensions $L/F$. 
\end{thm}
\begin{proof}
Let $(\{[b]_F, [a]_F\}, W)$ be admissible. Let $E=F(\sqrt{a},\sqrt{b})$. Let $L=E(\sqrt{W})$. 
Then $L/F$ is a Galois $D_8$-extension. (See for example \cite[Subsection 2.2]{MT3}.)

Now let $L/F$ be a Galois $D_8$-extension. We identify $D_8$ with $\U_3(\F_2)$. Let $\rho\colon {\rm Gal}(L/F)\to \U_3(\F_2)$ be any isomorphism. Set $\sigma_1=\rho^{-1}(E_{12})$, and $\sigma_2=\rho^{-1}(E_{23})$. Then the commutator subgroup $\Phi=[{\rm Gal}(L/F),{\rm Gal}(L/F)]$ is $\langle [\sigma_1,\sigma_2]\rangle.$

 Let $M$ be the fixed field of $\Phi$. Then $M/F$ is the an 2-elementary abelian extension of $F$, and ${\rm Gal}(M/F)$ is the internal direct sum 
\[
{\rm Gal}(M/F)=\langle \sigma_1|_M \rangle\oplus \langle \sigma_2|_M \rangle \simeq (\Z/2\Z)^2.
\]
Let $[a]_F,[b]_F$ be elements in $F^\times/(F^\times)^2$ which is dual to $\sigma_1|_M, \sigma_2|_M$ respectively via  Kummer theory. Explicitly we require that
\[
\begin{aligned}
\sigma_1(\sqrt{a})=-\sqrt{a},\sigma_1(\sqrt{b})=\sqrt{b};\\
\sigma_2(\sqrt{a})=\sqrt{a},\sigma_2(\sqrt{b})=-\sqrt{b}.
\end{aligned}
\]
\noindent
{\bf Claim:} $\{[b]_F,[a]_F\}$ does not depend on the choice of $\rho$.  

\noindent
{\it Proof of Claim:} Suppose that $\rho^\prime\colon {\rm Gal}(L/F)\to \U_4(\F_2)$ is another isomorphism. We define $\sigma^\prime_1={\rho^\prime}^{-1}(E_{12})$, and $\sigma^\prime_2={\rho^\prime}^{-1}(E_{23})$. We need to show that $\{\sigma_1|_M,\sigma_2|_M\}=\{\sigma^\prime_1|_M,\sigma^\prime_2|_M\}$. We first note that $\Phi$ is the center of ${\rm Gal}(L/F)$.

Because $\sigma_2^\prime|_M$ is in ${\rm Gal}(M/F)=\langle \sigma_1|_M \rangle\oplus \langle \sigma_2|_M\rangle$, we have that modulo the subgroup $\Phi$, $\sigma_2^\prime$ is equal to one of the following elements $\sigma_1,\sigma_2$, or $\sigma_1\sigma_2$. 

If $\sigma_2^\prime=\sigma_1\sigma_2$ modulo $\Phi$, then ${\sigma_2^\prime}^2=(\sigma_1\sigma_2)^2\not=1$, a contradiction. Similarly $\sigma_1^\prime$ cannot be $\sigma_1\sigma_2$ modulo $\Phi$.

{\bf Case 1:} $\sigma_2^\prime=\sigma_1$ modulo $\Phi$. In this case $\sigma_1^\prime$ cannot be $\sigma_1$ modulo $\Phi$. Otherwise it would lead to a contradiction that 
$1\not= [\sigma^\prime_1,\sigma^\prime_2]=[\sigma_1,\sigma_1]=1$. Hence $\sigma^\prime_1=\sigma_2$ modulo $\Phi$. 

{\bf Case 2:} $\sigma_2^\prime=\sigma_2$ modulo $\Phi$. In this case $\sigma_1^\prime$ cannot be $\sigma_2$ modulo $\Phi$. Otherwise it would lead to a contradiction that 
$1\not= [\sigma^\prime_1,\sigma^\prime_2]=[\sigma_2,\sigma_2]=1$. Hence $\sigma^\prime_1=\sigma_1$ modulo $\Phi$. 

In both cases we have $\{\sigma_1|_M,\sigma_2|_M\}=\{\sigma^\prime_1|_M,\sigma^\prime_2|_M\}$, as desired.
\\
\\  
We have an exact sequence 
\[
1 \to {\rm Gal}(L/F(\sqrt{a}))  \to {\rm Gal}(L/F) \to {\rm Gal}(F(\sqrt{a})/F) \to 1.
\]
Then ${\rm Gal}(L/F(\sqrt{a})$ is an $\F_2[{\rm Gal}(F(\sqrt{a})/F)]$-module where the action is by conjugation. We also have the ${\rm Gal}(F(\sqrt{a})/F)$-equivariant Kummer pairing 
\[
\frac{F(\sqrt{a})\cap (L^\times)^2}{(F(\sqrt{a})^\times)^2}\times {\rm Gal}(L/F(\sqrt{a})) \to \F_2.
\]
As an $\F_2$-vector space, ${\rm Gal}(L/F(\sqrt{a}))$ has a basis consisting of $\sigma_2$, $[\sigma_1,\sigma_2]$. 
Let $\delta$ be the element dual to $[\sigma_1,\sigma_2]$. 
Then ${\rm Nm}_{F(\sqrt{a})/F}(\delta)\equiv b \bmod (F(\sqrt{a})^{\times})^2$. Hence ${\rm Nm}_{F(\sqrt{a})/F}(\delta)$ is in $b(F^{\times})^2 \cup ba(F^{\times})^2$. 

Suppose that ${\rm Nm}_{F(\sqrt{a})/F}(\delta) = baf^2$,  for some $f\in F^\times$. From $\sigma_1(\delta)/\delta = ba (f/\delta)^2$, we see that
\[
\sigma_1(\sqrt{\delta})=(\pm) \sqrt{\delta} \sqrt{ba} f/\delta. 
\]
Hence
\[
\begin{aligned}
\sigma_1^2(\sqrt{\delta}) &= (\pm)\sigma_1(\sqrt{\delta}) \sigma_1(\sqrt{ba}) (f/\sigma_1(\delta))\\
&=(\pm)^2 \sqrt{\delta} \sqrt{ba} (f/\delta) \sqrt{b} (-\sqrt{a}) (f/\sigma_1(\delta))\\
&= -\sqrt{\delta}.
\end{aligned}
\]
This implies that $\sigma_1$ is not of order 2, a contradiction. Hence we have $[{\rm Nm}_{E/F}(\delta)]_F=[b]_F$. Let $W$ be the one dimensional $\F_2$-subspace of $M^\times/(M^\times)^2$ generated by $[\delta]_M$. Then $W$ is compatible with $\{[a]_F,[b]_F\}$. Also since $L= M(\sqrt{W})$, we see that $W$ does not depend on the choice of $\rho$.
\end{proof}

\begin{lem}
\label{lem:3}
 Assume that $F$ is a finite extension of $\Q_2$ of degree $n$. 
 Let $q$ be the highest power of $2$ such that $F$ contains a primitive $q$-th root of unity. 
Then the number $N$ of admissible unordered pairs $\{[a]_F,[b]_F\}$ is 
\[
\begin{cases} 
(2^{n+2}-1)(2^{n}-1) &\text{ if $q\not= 2$},\\
(2^{n+1}-1)^2 &\text{ if $q=2$}.
\end{cases}
\]
\end{lem}
\begin{proof}
Let $N^\prime$ be the number of $([a]_F,[b]_F)$ such that $(a,b)=0$ and that $\dim_{\F_2} \langle [a]_F,[b]_F\rangle=2$. Then $N=N^\prime/2$. On the other hand, by \cite[Remark 3.9]{MT2}, we have
\[
N^\prime=
\begin{cases} 
\displaystyle
(2^{n+2}-1)(2^{n+1}-2) &\text{ if $q\not= 2$},\\
2(2^{n+1}-1)^2 &\text{ if $q=2$}.
\end{cases}
\]
The result then follows.
\end{proof}

\begin{lem}
\label{lem:4}
 Assume that $F$ is a finite extension of $\Q_2$. Let us fix an unordered admissible pair $\{[a]_F,[b]_F\}$. Then the number of admissible pairs $(\{[a]_F,[b]_F\},W)$ is $2^{n}$. 
\end{lem}
\begin{proof}
By local class field theory we have an isomorphism
\[
\dfrac{F^\times}{{\rm Nm}_{F(\sqrt{a})/F}(F(\sqrt{a})^\times)} \simeq {\rm Gal}(F(\sqrt{a})/F)=\Z/2\Z.
\]
Since $G$ is of exponent 2, we see that $\dfrac{F^\times}{{\rm Nm}_{F(\sqrt{a})/F}(F(\sqrt{a})^\times)}$ is also of exponent 2. Hence 
\[ (F^\times)^2\subseteq {\rm Nm}_{F(\sqrt{a})/F}(F(\sqrt{a})^\times) \subseteq F^\times.\] 
Since $|G|=2$, we have 
\[
 2=\left| \dfrac{F^\times}{{\rm Nm}_{F(\sqrt{a})/F}(F(\sqrt{a})^\times)}\right|
= \left[\dfrac{F^\times}{(F^\times)^2} : \dfrac{{\rm Nm}_{F(\sqrt{a})/F}(F(\sqrt{a})^\times)}{(F^\times)^2}\right].
\]
By \cite[Chapter II, \S 5, Corollary 5.8]{Neu}, one has $|F^\times/(F^\times)^2|=2^{n+2}$. 
Hence 
\[\dfrac{{\rm Nm}_{F(\sqrt{a})/F}(E^\times)}{(F^\times)^2}=\left|\dfrac{F^\times}{(F^\times)^2}\right|/2=2^{n+1}\].

Consider the homomorphism ${\rm Nm}\colon \dfrac{F(\sqrt{a})^\times}{(F(\sqrt{a})^\times)^2}\to \dfrac{F^\times}{(F^\times)^2}$. Then $\im({\rm Nm})= \dfrac{{\rm Nm}_{E/F}(E^\times)}{(F^\times)^2}$.
By \cite[Chapter II, \S 5, Corollary 5.8]{Neu}, one has $|F(\sqrt{a})^\times/(F(\sqrt{a})^\times)^2|=2^{2n+2}$. Hence we have
\[
|\ker {\rm Nm}|= \left|\dfrac{F(\sqrt{a})^\times}{(F(\sqrt{a})^\times)^2}\right|/ |\im({\rm Nm})|= 2^{2n+2}/2^{n+1}=2^{n+1}.
\]
Hence
\[
| \{ [\delta]_{F(\sqrt{a})} \colon [{\rm Nm}_{F(\sqrt{a})/F}(\delta)]_F =[b]_F\}|= |\ker {\rm Nm}|=2^{n+1}. 
\]
Therefore the number of $W$ such that $(\{[b]_F,[a]_F\},W)$ is admissible, is $2^{n+1}/2=2^{n}$.
\end{proof}

We recover the following result, which was also obtained in \cite[Theorem 2.2]{Ya} (see also \cite[Theorem 11]{MNg}, \cite[Remark 3.9]{MT3}).
\begin{cor} 
Assume that $F$ is a finite extension of $\Q_2$ of degree $n$. Let $q$ be the highest power of $2$ such that $F$ contains a
primitive $q$-th root of unity. Then the number of Galois $D_8$-extensions of $F$ is 
\[
\begin{cases} 
2^n(2^{n+2}-1)(2^n-1) &\text{ if $q\not=2$},\\
2^n(2^{n+1}-1)^2 &\text{ if $q=2$}.
\end{cases}
\]
\end{cor}
\begin{proof} This follows from  Theorem~\ref{thm:D4 char not 2}, Lemma~\ref{lem:3} and Lemma~\ref{lem:4}. 
\end{proof}
\subsection{The case of characteristic 2}
Let $F$ be a field of characteristic 2. 
\begin{defn}
An unordered pair $\{[b]_F, [a]_F\}$, where $a$ and $b$ are in $F^\times$ is {\it admissible} if  $\dim_{\F_2}(\langle [a]_F,[b]_F\rangle)=2$.
\end{defn}

\begin{lem} 
Assume that $\{[b]_F,[a]_F\}$ is admissible. Let $E=F(\theta_a,\theta_b)$.  Then there exists $\delta_1\in F(\theta_{a})$ such that 
\[
[{\rm Tr}_{F(\theta_{a})/F}(\delta_1)]_F=[b]_F.
\]
Furthermore for any such $\delta_1$, there exists $\delta_2$ in $F(\theta_{b})$ such that ${[\delta_1]_E} =[\delta_2]_E$ and 
\[
[{\rm Tr}_{F(\theta_{b})/F}(\delta_2)]_F=[a]_F.
\]
\end{lem}
\begin{proof} 
As the trace map ${\rm Tr}_{F(\theta_a)/F}$ is surjective, there exists $\delta_1\in F(\theta_{a})$  such that 
\[
[{\rm Tr}_{F(\theta_{a})/F}(\delta_1)]_F=[b]_F.
\]

Now let $\delta_1$ be any element in $F(\theta_{a})$ such that $[{\rm Tr}_{F(\theta_{a})/F}(\delta_1)]_F=[b]_F$. We write $\delta_1= x+ y\theta_{a}$, where $x,y\in F$. Then $y=b+\wp(d)$, for some $d\in F$. We have
\[
\begin{aligned}
\delta_1+ x+ a\theta_b+ ab+ ad^2 &= x+ (b+\wp(d))\theta_a +x+ a\theta_b + ab +ad^2\\
&=  [b\theta_a + a\theta_b+ ab] +[ \wp(d)\theta_a+ ad^2]\\
&= [(\theta_a\theta_b)^2-\theta_a\theta_b] + [(d\theta_a)^2-d\theta_a].
\end{aligned}.
\]
Set $\delta_2=x+ a\theta_b+ ab+ ad^2 \in F(\theta_{b})$. Then $[\delta_1]_E=[\delta_2]_E$ and $[{\rm Tr}_{F(\theta_{b})/F}(\delta_2)]_F= [a]_F$. 
\end{proof}

The above lemma shows that the following definition is well-defined.
\begin{defn} Let $P=\{[b]_F, [a]_F\}$ be an admissible unordered pair. Let $E=F(\theta_{a},\theta_{b})$.

A one dimensional $\F_2$-subspace $W$ of $E/\wp(E)$ is said to be {\it compatible} with $P$ if $W$ is generated by a $\delta\in F(\theta_{a})$ with $[{\rm Tr}_{F(\theta_{a})/F}(\delta)]_F=[b]_F$.
In this case we say that $(P,W)$ is {\it admissible}.
\end{defn}

\begin{lem} Let $\{[a]_F,[b]_F\}$ be an admissible unordered pair Let $E=F(\theta_{a},\theta_{b})$. Let $\delta\in F(\theta_a)$ with $[{\rm Tr}_{F(\theta_{a})/F}(\delta)]_F=[b]_F$. Then $E(\theta_{\delta})/F$ is a Galois $D_8$-extension.
\end{lem}
\begin{proof}
The extension $E/F$ is Galois with Galois group generated by $\sigma_a,\sigma_b$, where $\sigma_a$ and $\sigma_b$ are defined by the conditions: 
\[
\begin{aligned}
\sigma_a(\theta_a)=\theta_a+1, \sigma_a(\theta_b)=\theta_b,\\
\sigma_b(\theta_a)=\theta_a, \sigma_b(\theta_b)=\theta_b+1,\\
\end{aligned}
\]
Since ${\rm Tr}_{F(\theta_a)/F}(\delta)=b+\wp(d)$ for some $d\in F$, we have
\[
\sigma_a(\delta)= \delta+ b+\wp(d).
\]
Clearly we have 
\[
\sigma_b(\delta)=\delta.
\]

Then \cite[Proof of Proposition 4.1]{MT3} shows that $L=E(\theta_\delta)/F$ is Galois and its Galois group is isomorphic to $D_8$. Furthermore, we can choose an extension, still denoted $\sigma_a$ in ${\rm Gal}(L/F)$, of $\sigma_a$ such that $\sigma_a(\theta_\delta)=\theta_\delta+\theta_b+d$.
\end{proof}
\begin{thm}
\label{thm:D4 char 2}
  Let $F$  be a field of characteristic 2.
There is a natural one-to-one correspondence between the set of admissible pairs $(\{[a]_F,[b]_F\},W)$ and the set of Galois $D_8$ extensions $L/F$. 
\end{thm}
\begin{proof}
Let $(\{[b]_F, [a]_F\}, W)$ be admissible. Let $E=F(\sqrt{a},\sqrt{b})$. Let $L=E(\sqrt{W})$. Then $L/F$ is a Galois $D_8$-extension. (See \cite[Subsection 4.2]{MT3}.)

Now let $L/F$ be a Galois $D_8$-extension. We identify $D_8$ with $\U_3(\F_2)$. Let $\rho\colon {\rm Gal}(L/F)\to \U_3(\F_2)$ be any isomorphism. Set $\sigma_1=\rho^{-1}(E_{12})$, and $\sigma_2=\rho^{-1}(E_{23})$. Then the commutator subgroup $\Phi=[{\rm Gal}(L/F),{\rm Gal}(L/F)]$ is $\langle [\sigma_1,\sigma_2]\rangle.$

 Let $M$ be the fixed field of $\Phi$. Then $M/F$ is the an 2-elementary abelian extension of $F$, and ${\rm Gal}(M/F)$ is the internal direct sum 
\[
{\rm Gal}(M/F)=\langle \sigma_1|_M \rangle\oplus \langle \sigma_2|_M \rangle \simeq (\Z/2\Z)^2.
\]
Let $[a]_F,[b]_F$ be elements in $F/\wp(F)^2$ which is dual to $\sigma_1|_M, \sigma_2|_M$ respectively via  Artin-Schreier theory. Explicitly we require that
\[
\begin{aligned}
\sigma_1(\theta_{a})=\theta_{a}+1,\sigma_1(\theta_{b})=\theta_{b};\\
\sigma_2(\theta_{a})=\theta_{a},\sigma_2(\theta_{b})=-\theta_{b}.
\end{aligned}
\]
\noindent
{\bf Claim:} $\{[b]_F,[a]_F\}$ does not depend on the choice of $\rho$.  

\noindent
{\it Proof of Claim:} Suppose that $\rho^\prime\colon {\rm Gal}(L/F)\to \U_4(\F_2)$ is another isomorphism. We define $\sigma^\prime_1={\rho^\prime}^{-1}(E_{12})$, and $\sigma^\prime_2={\rho^\prime}^{-1}(E_{23})$. We need to show that $\{\sigma_1|_M,\sigma_2|_M\}=\{\sigma^\prime_1|_M,\sigma^\prime_2|_M\}$. We first note that $\Phi$ is the center of ${\rm Gal}(L/F)$.

Because $\sigma_2^\prime|_M$ is in ${\rm Gal}(M/F)=\langle \sigma_1|_M \rangle\oplus \langle \sigma_2|_M\rangle$, we have that modulo the subgroup $\Phi$, $\sigma_2^\prime$ is equal to one of the following elements $\sigma_1,\sigma_2$, or $\sigma_1\sigma_2$. 

If $\sigma_2^\prime=\sigma_1\sigma_2$ modulo $\Phi$, then ${\sigma_2^\prime}^2=(\sigma_1\sigma_2)^2\not=1$, a contradiction. Similarly $\sigma_1^\prime$ cannot be $\sigma_1\sigma_2$ modulo $\Phi$.

{\bf Case 1:} $\sigma_2^\prime=\sigma_1$ modulo $\Phi$. In this case $\sigma_1^\prime$ cannot be $\sigma_1$ modulo $\Phi$. Otherwise it would lead to a contradiction that 
$1\not= [\sigma^\prime_1,\sigma^\prime_2]=[\sigma_1,\sigma_1]=1$. Hence $\sigma^\prime_1=\sigma_2$ modulo $\Phi$. 

{\bf Case 2:} $\sigma_2^\prime=\sigma_2$ modulo $\Phi$. In this case $\sigma_1^\prime$ cannot be $\sigma_2$ modulo $\Phi$. Otherwise it would lead to a contradiction that 
$1\not= [\sigma^\prime_1,\sigma^\prime_2]=[\sigma_2,\sigma_2]=1$. Hence $\sigma^\prime_1=\sigma_1$ modulo $\Phi$. 

In both cases we have $\{\sigma_1|_M,\sigma_2|_M\}=\{\sigma^\prime_1|_M,\sigma^\prime_2|_M\}$, as desired.
\\
\\  
We have an exact sequence 
\[
1 \to {\rm Gal}(L/F(\theta_{a}))  \to {\rm Gal}(L/F) \to {\rm Gal}(F(\theta_{a})/F) \to 1.
\]
Then ${\rm Gal}(L/F(\theta_{a})$ is an $\F_2[{\rm Gal}(F(\theta_{a})/F)]$-module where the action is by conjugation. We also have the ${\rm Gal}(F(\theta_{a})/F)$-equivariant Artin-Schreier pairing 
\[
\frac{F(\theta_{a})\cap \wp(L)}{\wp(F(\theta_{a}))}\times {\rm Gal}(L/F(\theta_{a})) \to \F_2.
\]
As an $\F_2$-vector space ${\rm Gal}(L/F(\theta_{a}))$ has a basis consisting of $\sigma_2$, $[\sigma_1,\sigma_2]$. 
Let $\delta$ be the element dual to $[\sigma_1,\sigma_2]$. 
Then ${\rm Tr}_{F(\theta_{a})/F}(\delta)\equiv b \bmod (\wp(F(\theta_{a}))$. Hence ${\rm Nm}_{F(\theta_{a})/F}(\delta)$ is in $b+\wp(F) \cup b+a+\wp(F)$. 

Suppose that ${\rm Tr}_{F(\theta_{a})/F}(\delta) = b+a+\wp(f)$,  for some $f\in F$. Then $\sigma_1(\delta)=\delta + b +a +\wp(f)$. Thus
\[
\sigma_1(\theta_{\delta})= \theta_{\delta}+ \theta{b}+ \theta_a + f +i, 
\]
for some $i\in \{0,1\}$.
Hence
\[
\begin{aligned}
\sigma_1^2(\theta_{\delta}) &= \sigma_1(\theta_{\delta})+ \sigma_1(\theta_{b})+\sigma_1(\theta_a)+ f+i\\
&= \theta_{\delta}+ \theta_b+\theta+a+ f+i \theta_{b}+ \theta_a +1+ f+i\\
&= \theta_{\delta}+1.
\end{aligned}
\]
This implies that $\sigma_1$ is not of order 2, a contradiction. Hence we have $[{\rm Tr}_{E/F}(\delta)]_F=[b]_F$. Let $W$ be the one dimensional $\F_2$-subspace of $M^\times/(M^\times)^2$ generated by $[\delta]_M$. Then $W$ is compatible with $\{[a]_F,[b]_F\}$. Also since $L= M(\theta_{W})$, we see that $W$ does not depend on the choice of $\rho$.
\end{proof}

\section{Description of $\U_4(\F_2)$-extensions: The case of characteristic 2} 

Let $F$ be a field of characteristic  2. 
\begin{defn}
A pair $([b]_F, V)$ where $b$ is in $F$ and $V\subseteq F/\wp(F)$ is {\it admissible} if $\dim_{\F_2}(V)=2$ and $\dim_{\F_2}(\langle V,[b]_F\rangle)=3$.
\end{defn}

\begin{lem} Assume that $([b]_F,V)$ is admissible. Let $E=F(\wp^{-1}(V))$.  Then there exists $\delta\in E$ such that $[{\rm Tr}_{E/F}(\delta)]_F=[b]_F$. 
\end{lem}
\begin{proof} It is clear since we know that the trace map ${\rm Tr}_{E/F}$ is surjective.
\end{proof}
\begin{defn} 
Assume that $([b]_F,V)$ is admissible. Let $E=F(\wp^{-1}(V))$. Then a triple $([b]_F, V, W)$ where $W$ is a free  $\F_2[{\rm Gal}(E/F)]$-submodule of $E/\wp(E)$, is {\it admissible} if $W$ is generated by an element $[\delta]_E$ with $[{\rm Tr}_{E/F}(\delta)]_F =[b]_F$.
\end{defn}

\begin{lem}
Let $([b]_F, V, W)$ be an admissible triple. Assume that $V=\langle [a]_F,[c]_F\rangle$. Let $E=F(\wp^{-1}(V))$. Assume that $W$ is generated by $[\delta]_E$ as a free $\F_2[{\rm Gal}(E/F)]$-module with $[{\rm Tr}_{E/F}(\delta)]_F =[b]_F$. Let $A= {\rm Tr}_{E/F(\theta_a)}(\delta)$ and $C={\rm Tr}_{E/F(\theta_c)}(\delta)$. Then every generator of $W$ as a free $\F_2[{\rm Gal}(E/F)]$-module is of the form 
\[[\delta^\prime]_E= [\delta]_E+ \epsilon_A[A]_E + {\epsilon_C}[C]_E +{\epsilon_b} [b]_E,\]
where $\epsilon_A,\epsilon_C,\epsilon_b\in \{0,1\}$.

Furthermore for any generator $[\delta^\prime]_E$ of $W$ as a free $\F_2[{\rm Gal}(E/F)]$-module, we have $[{\rm Tr}_{E/F}(\delta^\prime)]_F=[b]_F$. In particular, this implies that the pair $(V,W)$ uniquely determines  $[b]_F$. 
\end{lem}
\begin{proof} Let $G={\rm Gal}(E/F)$. 
As an $\F_2$-vector space, $W$ is generated  by $[\delta]_E, [A]_E, [C]_E, [b]_E$.
Let $[\delta^\prime]_E$ be an arbitrary generator of the free $\F_2[G]$-module. Then 
\[[\delta^\prime]_E= {\epsilon_\delta}[\delta]_E + {\epsilon_A}[A]_E + {\epsilon_C}[C]_E +  {\epsilon_b}[b]_E,\]
for some $\epsilon_{\delta}, \epsilon_A,\epsilon_C,\epsilon_b\in \{0,1\}$.
Suppose that $\epsilon_{\delta}=0$, then we see that $(\sum_{\sigma\in G}\sigma)([\delta^\prime]_E)$ is trivial in $E/\wp((E))$, a contradiction. Hence $\epsilon_{\delta}=1$.

Conversely, assume that  $[\delta^\prime]_E=[\delta]_E + {\epsilon_A}[A]_E+ {\epsilon_C}[C]_E+ {\epsilon_b}[b]_E,$ for some $\epsilon_A,\epsilon_C,\epsilon_b\in \{0,1\}$. Let $W^\prime$ be the $\F_2[G]$-module generated by $[\delta^\prime]_E$. Then we have $W^\prime\subseteq W$. It is then enough to show that $W^\prime$ is a free $\F_2[G]$-module. Suppose that $W^\prime$ would not be free. Then there would exist a non-zero ideal $I\subseteq \F_2[G]$ such that $I$ would annihilate $\delta^\prime$. But it is known that any non-zero ideal of $\F_2[G]$ contains the element $\sum_{\sigma\in G}\sigma=:N$. Therefore $N$ would annihilate $[\delta^\prime]_E$. This contradicts to the fact that
\[
N([\delta^\prime]_E)]=[{\rm Tr}_{E/F}(\delta^\prime)]_E= [b]_E \not=0\in E/\wp(E).
\qedhere
\] 
\end{proof}

\begin{prop}
\label{prop:admissible triple char 2}
 Let  $([b]_F, V, W)$ be an admissible triple. Let $E=F(\wp^{-1}{V})$. Let $L=E(\wp^{-1}{W})$. Then  $L/F$ is a Galois $\U_4(\F_2)$-extension.
\end{prop}
\begin{proof}
Suppose that $V=\langle [a]_F,[c]_F\rangle$ and that $W$ is generated by $\delta$ with ${\rm Tr}_{E/F}(\delta)=b+\wp(d)$, for some $d\in F$. Let $A= {\rm Tr}_{E/F(\theta_{a})}(\delta)$ and $C={\rm Tr}_{E/F(\theta_{c})}(\delta)$. 
We first note that $F(\theta_{a},\theta_{b},\theta_{c})/F$ is an abelian 2-elementary extension whose Galois group is generated by $\sigma_a,\sigma_b,\sigma_c$, where
\[
\begin{aligned}
\sigma_a(\theta_{a})&= \theta_{a}+1, \sigma_a(\theta_{b})=\theta_{b}, \sigma_a(\theta_{c})=\theta_{c};\\
\sigma_b(\theta_{a})&=\theta_{a}, \sigma_b(\theta_{b})= \theta_{b}+1, \sigma_b(\theta_{c})= \theta_{c};\\
\sigma_c(\theta_{a})&=\theta_{a}, \sigma_c(\theta_{b})=\theta_{b}, \sigma_c(\theta_{c})=\theta_{c}+1.
\end{aligned}
\]

Clearly we have
\[
\begin{aligned}
\sigma_c(\delta)&=\delta + A,\\
\sigma_a(\delta)&=\delta + C , \\
\sigma_a(A) &= A +b +\wp(d),\\
\sigma_c(C) &= C+ b+ \wp{d}.\\
\end{aligned}
\]
Then \cite[Proof of Theorem 4.2]{MT3} shows that $L/F$ is a Galois $\U_4(F_p)$-extension. Moreover an explicit isomorphism $\rho\colon {\rm Gal}(L/F)\to \U_4(\F_2)$ is given by
\[
\sigma_a \mapsto E_{12}, \; \;
\sigma_b\mapsto E_{23}, \;\;
 \sigma_c\mapsto E_{34},
\]
for suitable extensions  $\sigma_a,\sigma_b,\sigma_c\in {\rm Gal}(L/F)$ of $\sigma_a,\sigma_b,\sigma_c$.
\end{proof}

\begin{prop}
\label{prop:Galois extension char 2}
There is a natural way to associate an admissible triple $([b]_F,V,W)$ to any given Galois $\U_4(\F_2)$-extension $L/F$.
\end{prop}
\begin{proof}
 Let $\rho\colon {\rm Gal}(L/F)\to \U_4(\F_2)$ be any isomorphism. Set $\sigma_1=\rho^{-1}(E_{12})$, $\sigma_2=\rho^{-1}(E_{23})$, and $\sigma_3=\rho^{-1}(E_{34})$. Then the commutator subgroup $\Phi=[{\rm Gal}(L/F),{\rm Gal}(L/F)]$ is the internal direct sum 
\[
\Phi=\langle [\sigma_1,\sigma_2]\rangle \oplus \langle [\sigma_2,\sigma_3]\rangle \oplus \langle [[\sigma_1,\sigma_2],\sigma_3]\rangle \simeq (\Z/2\Z)^3.
\]

 Let $M$ be the fixed field of $\Phi$. Then $M/F$ is an abelian 2-elementary extension of $F$, and ${\rm Gal}(M/F)$ is the internal direct sum 
\[
{\rm Gal}(M/F)=\langle \sigma_1|_M \rangle\oplus \langle \sigma_2|_M\rangle \oplus \langle \sigma_3|_M\rangle \simeq (\Z/2\Z)^3.
\]
Let $[a]_F,[b]_F,[c]_F$ be elements in $F/\wp(F)$ which is dual to $\sigma_1|_M, \sigma_2|_M, \sigma_3|_M$ respectively via  Artin-Schreier theory. Explicitly we require that
\[
\begin{aligned}
\sigma_1(\theta_a)=\theta_{a}+1,\sigma_1(\theta_{b})=\theta_{b}, \sigma_1(\theta_{c})=\theta_{c};\\
\sigma_2(\theta_{a})=\theta_{a},\sigma_2(\theta_{b})=\theta_{b}+1, \sigma_2(\theta_{c})=\theta_{c};\\
\sigma_3(\theta_{a})=\theta_{a},\sigma_3(\theta_{b})=\theta_{b}, \sigma_3(\theta_{c})=\theta_{c}+1. 
\end{aligned}
\]
Let $E=F(\theta_{a},\theta_{c})$. Then $E$ is fixed under $\sigma_2$, $[\sigma_1,\sigma_2]$, $[\sigma_2,\sigma_3]$ and $[[\sigma_1,\sigma_2],\sigma_3]$. 
 Hence $E$ is fixed under a subgroup $H$ of ${\rm Gal(L/F)}$ which is generated by $\sigma_2$, $[\sigma_1,\sigma_2]$, $[\sigma_2,\sigma_3]$ and $[[\sigma_1,\sigma_2],\sigma_3]$. We have $[L^H:F]=|{\rm Gal}(L/F)|/|H|=4$,  and $[E:F]=4$. Therefore $E=L^H$.
\\
\\ 
{\bf Claim:} $E$ does not depend on the choice of $\rho$.  

\noindent
{\it Proof of Claim:} Suppose that $\rho^\prime\colon {\rm Gal}(L/F)\to \U_4(\F_2)$ is another isomorphism. We define $\sigma^\prime_1={\rho^\prime}^{-1}(E_{12})$, $\sigma^\prime_2={\rho^\prime}^{-1}(E_{23})$, and $\sigma^\prime_3={\rho^\prime}^{-1}(E_{34})$. Let $H^\prime$ be the group generated by $\sigma_2^\prime$, $[\sigma_1^\prime,\sigma_2^\prime]$, $[\sigma_2^\prime,\sigma_3^\prime]$ and $[[\sigma_1^\prime,\sigma_2^\prime],\sigma_3^\prime]$. We need to show that $H=H^\prime$. We first note that $\sigma_2$ and $\sigma_2^\prime$ commute with every element in $\Phi$.

Clearly $\sigma_2^\prime|_M$ is in ${\rm Gal}(M/F)=\langle \sigma_1|_M \rangle\oplus \langle \sigma_2|_M\rangle \oplus \langle \sigma_3|_M\rangle$. 

Hence modulo the subgroup $\Phi$, $\sigma_2^\prime$ is equal to one of the following elements $\sigma_1,\sigma_2,\sigma_3$, $\sigma_1\sigma_2,\sigma_1\sigma_3,\sigma_2\sigma_3$, $\sigma_1\sigma_2\sigma_3$.

If $\sigma_2^\prime=\sigma_1,$ or $\sigma_1\sigma_2$, or $\sigma_1\sigma_3$, or $\sigma_1\sigma_2\sigma_3$ modulo $\Phi$, then 
\[
[[\sigma_2,\sigma_3],\sigma_2^\prime]=[[\sigma_2,\sigma_3],\sigma_1],
\] 
which is impossible since $[[\sigma_2,\sigma_3],\sigma_1]$ is nontrivial but $[[\sigma_2,\sigma_3],\sigma^\prime_2]$ is trivial.

If $\sigma_2^\prime=\sigma_3,$ or $\sigma_2\sigma_3$ modulo $\Phi$, then 
\[
[[\sigma_1,\sigma_2],\sigma_2^\prime]=[[\sigma_1,\sigma_2],\sigma_3],
\] 
which is impossible since $[[\sigma_1,\sigma_2],\sigma_3]$ is nontrivial but $[[\sigma_1,\sigma_2],\sigma^\prime_2]$ is trivial.

From the above discussion we see that $\sigma^\prime_2\equiv \sigma_2\bmod \Phi$. This implies that $[b]_F$ does not depend on the choice of $\rho$ 
and that $H^\prime=H$. Thus $E$ does not depend on the choice of $\rho$ also.
\\
\\  
We have an exact sequence 
\[
1 \to {\rm Gal}(L/E)  \to {\rm Gal}(L/F) \to {\rm Gal}(E/F)=G \to 1.
\]
Then ${\rm Gal}(L/E)$ is an $\F_2[G]$ module where the action is by conjugation. We also have the $G$-equivariant Artin-Schreier pairing 
\[
\frac{E\cap \wp(L  )}{\wp(E)}\times {\rm Gal}(L/E) \to \F_2.
\]
As an $\F_2$-vector space, ${\rm Gal}(L/E)$ has a basis consisting of $\sigma_2$, $[\sigma_1,\sigma_2]$, $[\sigma_2,\sigma_3]$ and $[[\sigma_1,\sigma_2],\sigma_3]$. 
Let $[\delta]_E$ be an element dual to $[[\sigma_1,\sigma_2],\sigma_3]$. 
Then ${\rm Tr}_{E/F}(\delta)\equiv b \bmod \wp(E)$. Hence ${\rm Tr}_{E/F}(\delta)$ is in $(b+\wp(F)) \cup (b+a+\wp(F))\cup (b+c+\wp(F)) \cup (b+a+c+\wp(F)^2)$. 

Let $A={\rm Tr}_{E/F(\theta_{a})}(\delta)$ and $C={\rm Tr}_{E/F(\theta_{c})}(\delta)$. 
Suppose that ${\rm Tr}_{E/F}(\delta)\equiv b+a \mod \wp(F)$. Then ${\rm Tr}_{F(\theta_{a})/F}(A)= b+a+ \wp(f)$ for some $f\in F$. Hence $\sigma_1(A) = A+b+a+\wp(f)$. Thus 
\[
\sigma_1(\theta_{A})= \theta_{A}+ \theta_b+\theta_a +f +i, 
\]
for some $i\in \{0,1\}$.
Therefore
\[
\begin{aligned}
\sigma_1^2(\theta_{A}) &= \sigma_1(\theta_{A})+ \sigma_1(\theta_{b})+\sigma_1(\theta_a)+ f +i\\
&= \theta_{A}+ \theta_b+ \theta_{a} +f +i + \theta_b+ \theta_a+1 + f+i\\
&= \theta_{A}+1.
\end{aligned}
\]
This implies that $\sigma_1$ is not of order 2, a contradiction. Hence we have ${\rm Nm}_{E/F}(\delta)$ is not in  $b+a+\wp(F)$.

Similarly we can show that ${\rm Nm}_{E/F}(\delta)$ is not in $(b+c+ \wp(F)) \cup (b+a+\wp(F)$. Therefore 
\[
{\rm Tr}_{E/F}(\delta)\equiv b \bmod \wp(F).
\]

 We set $V=\langle [a]_F,[c]_F\rangle$. Then $V$ does not depend on the choice of $\rho$, and the pair $([b]_F,V)$ is admissible.
Let $W$ be the $\F_2[G]$-submodule of $E/\wp(E$ which is dual via  Artin-Schreier theory to ${\rm Gal}(L/E)$. Then $W$  does not depend on the choice of $\rho$, and $W$ is free and generated by $\delta$. Since $[{\rm Tr}_{E/F}(\delta)]_F=[b]_F$, we see that the triple $([b]_F,V,W)$ is admissible. Since $[b]_F$ is uniquely determined by $(V,W)$, we see that $[b]_F$ does not depend on the choice of $\rho$.

\end{proof}

\begin{thm}
\label{thm:U4 char 2}
 Let $F$  be a field of characteristic 2.
There is a natural one-to-one correspondence between the set of admissible triples $([b]_F,V,W)$ and the set of Galois $\U_4(\F_2)$ extensions $L/F$. 
\end{thm}

\begin{proof}By Proposition~\ref{prop:admissible triple char 2} we have a map $\mu$ from the set of admissible triples $([b]_F,V,W)$ to the set of Galois $\U_4(\F_2)$-extensions $L/F$.
By Proposition~\ref{prop:Galois extension} we have a map $\eta$ from the set of Galois $\U_4(\F_2)$-extensions $L/F$ to  the set of admissible triples  $([b]_F,V,W)$. We show that $\mu$ and $\eta$ are the inverses of each other.

Let $([b]_F,V,W)$ be an admissible triple. Via the map $\mu$ we obtain a $\U_4(\F_2)$-extension $L/F$. Explicitly, if $V=\langle [a]_F,[c]_F\rangle$ and $E=F(\sqrt{a},\sqrt{c})$, then $L=E(\sqrt{W})$ and there is an isomorphism $\rho\colon {\rm Gal}(L/F)\simeq \U_4(\F_2)$ such that
$\rho^{-1}(E_{12})=\sigma_a$, $\rho^{-1}(E_{23})=\sigma_{b}$, $\rho^{-1}(E_{34})=\sigma_c$. (Here $\sigma_a,\sigma_b,\sigma_c$ are defined as in Proposition~\ref{prop:admissible triple}.)
We apply the construction in Proposition~\ref{prop:Galois extension char 2} with this isomorphism $\rho$. Then we obtain back the admissible triple $([b]_F,V,W)$.

Now let $L/F$ be a $\U_4(\F_2)$-extension. Then via the map $\eta$ we obtain an admissible triple $([b]_F,V,W)$. Since $L=F(\sqrt{V})(W)$, we see that $\mu$ sends the triple $([b]_F,V,W)$ back to the extension $L/F$
\end{proof}

\begin{lem}  Assume that $\dim_{\F_2}(F/\wp(F))=n<\infty$. Then the number $N$ of admissible pairs $([b]_F,V)$ is $\dfrac{4(2^n-1)(2^{n-1}-1)(2^{n-2}-1)}{3}$.
\end{lem}
\begin{proof} Recall that the Gaussian binomial coefficients are defined by
\[
\binom{n}{r}_q=
\begin{cases}
\dfrac{(q^n-1)(q^{n-1}-1)\cdots (q^{n-r+1}-1)}{(q-1)(q^2-1)\cdots (q^r-1)} &\text{ if $r\leq n$}\\
0 &\text{ if $r>n$}.
\end{cases}
\]

Every admissible pairs $([b]_F,V)$ can be obtained as follows. First, we  choose a three dimensional $\F_2$-subspace $V^\prime$ of $F/\wp(F)$. The number of choices of such $V^\prime$ is $\binom{n}{3}_2$. Then we choose a two dimensional $\F_2$-subspace $V$ of $V^\prime$. The number of choices of such $V$ is $\binom{3}{2}_2$. Finally, we choose a vector $[b]_F$ in $V^\prime\setminus V$. The number of choices of such $b$ is $8-4=4$. Therefore we have 
\[
N=\binom{n}{3}_2\times \binom{3}{2}_2\times 4=\dfrac{4(2^n-1)(2^{n-1}-1)(2^{n-2}-1)}{3}.
\qedhere
\]
\end{proof}

\begin{lem}
Assume that $\dim_{\F_2}(F/\wp(F))=n<\infty$. Let $([b]_F,V)$ be a fixed admissible pair. Then $n\geq 3$ and the number of admissible triples $([b]_F,V,W)$ is $2^{3n-6}$. 
\end{lem}
\begin{proof} Since there exists at least one admissible pair, namely $([b]_F,V)$, we see that $n\geq 3$. 

It is known that for a field $L$ of characteristic $2$, then the maximal pro-$2$-quotient $G_L(2)$ of the absolute Galois group of $L$ is free of rank $\dim_{\F_2}(L/\wp(L))$.

Let $E=F(\wp^{-1}(V))$. Then $G_E(2)$ is a (closed) subgroup of index $4$ in the free pro-$2$-group $G_F(2)$ of rank $n$. Thus $G_E(2)$ is also free and of rank $4n-3$.

Consider the surjective homomorphism ${\rm Tr}\colon \dfrac{E}{\wp(E)}\to \dfrac{F}{\wp(F)}$. We have
\[
|\ker ({\rm Tr})|= \left|\dfrac{E}{\wp(E)}\right|/ \left|\dfrac{F}{\wp(F)}\right|= 2^{4n-3}/2^n=2^{3n-3}.
\]
Hence
\[
| \{ [\delta]_E \colon [{\rm Tr}_{E/F}(\delta)]_F =[b]_F\}|= |\ker {\rm Tr}|=2^{3n-3}. 
\]
Therefore the number of $W$ such that $([b]_F,V,W)$ is admissible, is $2^{3n-3}/8=2^{3n-6}$.
\end{proof}
\begin{cor} Assume that $\dim_{\F_2}(F/\wp(F))=n<\infty$. Then the number of Galois $\U_4(\F_2)$-extensions $L/F$ is $\dfrac{(2^n-1)(2^{n-1}-1)(2^{n-2}-1)2^{3n-4}}{3}$.
\end{cor}

In the next proposition we show in particular that for each natural number $n$ there exist a field satisfying the hypothesis of the above corollary.
\begin{prop} Let $p$ a prime number. Then for each cardinal number $\sC$ there exists a field  $K$ of characteristic $p$ such that $[K: \wp(K)] = \sC$.
\end{prop}
\begin{proof}
Consider any $\F_p$-vector space $V$ such that $\dim_{\F_p}(V)$ is $\sC$. Let $V^*={\rm Hom}(G,\Q/\Z)$ be the Pontrjagin dual of $V$.
Then $V^*$ is a profinite (abelian) group. 
By \cite[Theorem 2]{Wa1} there exists a field $F$ of characteristic $p$ such that $F$ admits a Galois extension $L/F$ with ${\rm Gal}(L/F) = V^*$.
By  Artin-Schreier theory we conclude that ${\rm Hom}_ {cont}(V^*,\F_p) = H^1(V^*,\F_p)$, which is isomorphic canonically with $V$ via Pontrjagin duality, is isomorphic to $A/(\wp(F)$ ,where $A$ is some subgroup of $F$ containing $\wp(F)$. Hence the $\F_2$-dimension of $A/\wp(F)$ is $\sC$.

Now consider the maximal Galois extension $K/F$ in the maximal $p$-extension $F(p)$ of $F$ such that:
(*) the natural map $A/\wp(F) \to K/\wp(K)$ is an injection.

{\bf Claim 1:} Such an extension $K/F$ exists. 

{\it Proof}: .
Let $\sS$ be the set of all fields extension $K$ over $F$ in $F(p)$ satisfying the condition (*). Then $\sS$ is not empty since it contains at least  $F$. This set is partially ordered by set inclusion. We shall apply apply Zorn's lemma. We take a non-empty totally ordered subset $\sT$ of $\sS$. Let $K$ be the union of all fields $K_i$ in $\sT$. Clearly $K/F$ is a field extension and $K\subseteq F(p)$. Consider the natural map $A\wp(K)\to K/\wp(K)$. Suppose that this map is not injective. Then $A\cap \wp(K)$ is strictly larger than $\wp(F)$. However $A\cap \wp(K)=\bigcup_{K_i\in \sT}(A\cap \wp(K_i))$. Thus there exists a field $K_i\in \sT$ such that 
$A\cap \wp(K_i)$ is strictly larger than $\wp(F)$. This implies that that the natural map $A/\wp(F)\to K_i/\wp(K_i)$ is not injective, which contradicts  the condition that $K_i$ satisfies (*).
Therefore the map $A\wp(K)\to K/\wp(K)$ is injective and $K$ is in $\sT$. Clearly $K$ is greater than every element in $\sT$. 
The Claim then follows from Zorn's lemma.

{\bf Claim 2:} The above injection 
$A/\wp(F) \to K/\wp(K)$
is an isomorphism.

{\it Proof}: If the injection is not an isomorphism, then there exists an element $u$ in $K$ such that $u\not \equiv a\bmod \wp(K)$ for every $a\in A$. We have $A\cap (iu+\wp(K))=\emptyset$ for every $i=1,2\ldots,p-1$. Let $T= K(\theta_u)$. Then $T$ is strictly larger than $K$ and $T\subseteq F(p)$.
We have
\[
\begin{aligned}
A\cap \wp(T)&= A\cap (K\cap \wp(T))
 = A \cap [\bigcup_{i=0}^{p-1}(iu+\wp(K)]
 = A\cap \wp(K)
 =\wp(F).
\end{aligned}
\]

We consider the natural map 
$
\eta\colon A/\wp(F) \to T/\wp(T). 
$
Then $\ker(\eta)= \dfrac{A\cap\wp(T)}{\wp(F)}=0$. Thus $\eta$ is an injective. This contradicts  the maximality of $K$.
\end{proof}

\section{Example: The case $F=\Q_2$}

In this section we illustrate our results by considering the case that the base field is the field $\Q_2$ of 2-adic numbers. Here we provide a list of all unipotent Galois extensions $L/\Q_2$ with Galois groups isomorphic to $\U_n(\F_2)$ for $n\geq 2$. This completes the work of Naito (\cite{Na}) who listed all dihedral extensions of order 8 over $\Q_2$. 
The actual checking that our list is the complete list of all $\U_n(\F_2)$- Galois extensions of $\Q_2$ still requires  some work.
However because it is a straightforward application of the theory of Galois unipotent extensions in our paper, we omit basic numerical verifications.
The field $\Q_2$ has rather special role in Galois theory. Historically it attracted attention in work of Demushkin, Labute, Serre, Shafarevich and Weil. (See for example \cite{La},\cite{Sha},\cite{Se},\cite{We}.)

Assume that $F$ is $\Q_2$. Then we know that $[-1], [2], [5]$ is a basis for the $\F_2$-vector space $\Q_2^\times/(\Q_2^\times)^2$. (Here for simplicity, we denote $[a]$ for the class of $a$ in $\Q_2^\times/\Q_2^\times$.) The maximal abelian 2-elementary extension $K$ of $\Q_2$ is $\Q_2(\sqrt{-1},\sqrt{2},\sqrt{5})$.
\begin{prop} There are no Galois $\U_n(\F_2)$-extensions over $\Q_2$ for every $n\geq 5$. 
\end{prop}
\begin{proof}   Suppose that there is a Galois extension $L/\Q_2$ with Galois group isomorphic to $\U_n(\F_2)$ for some $n\geq 5$. Then we have a surjective homorphism $\rho\colon {\rm Gal}_{\Q_2}\to \U_n(\F_p)$. The homomorphism 
\[
\varphi=(\rho_{12},\ldots,\rho_{n-1,n})\colon G\to \F_p\times \cdots \times \F_p
\]
induced by the projection of $\U_n(\F_p)$ on its near-by diagonal is also surjective. Let $N$ be the fixed field under the subgroup $\ker(\varphi)$. Then $K/\Q_2$ is an abelian 2-extension with ${\rm Gal}(K/\Q_2)\simeq (\Z/2\Z)^{n-1}$. This implies that $N$ is contained in the maximal abelian 2-extension $K$ of $\Q_2$. But this contradicts to  the fact that $[N:\Q_2]=2^{n-1}>8=[K:\Q_2]$. 
\end{proof}
\subsection{A list of $\U_2(\F_2)$-extensions of $\Q_2$} Here is a list of Galois $\U_2(\F_2)=\Z/2\Z$ extensions of $\Q_2$: $\Q_2(\sqrt{-1})$, $\Q_2(\sqrt{2})$, $\Q_2(\sqrt{5})$, $\Q_2(\sqrt{-2})$, $\Q_2(\sqrt{-5})$, $\Q_2(\sqrt{10})$, $\Q_2(\sqrt{-10})$.

\subsection{A list of $\U_3(\F_2)$-extensions of $\Q_2$} Here is a list of Galois $\U_3(\F_2)=D_8$ extensions of $\Q_2$. Here a pair $\{[a],[b]\}$ in the first column is an unordered admissible pair which we refer to Theorem~\ref{thm:D4 char not 2}. Here we have 9 unordered admissible pairs $\{[a], [b]\}$ and each gives rise to two further admissible pairs $(\{[a], [b]\},W)$.
\begin{itemize}
\item $\{[-1], [2]\}$: $\mathbb{Q}_2(\sqrt{1+\sqrt{2}},\sqrt{-1}), \mathbb{Q}_2(\sqrt{3+\sqrt{2}},\sqrt{-1})$;
\item $\{[-1], [5]\}$:  $\mathbb{Q}_2(\sqrt{2+\sqrt{5}},\sqrt{-1})$, $\mathbb{Q}_2(\sqrt{2(2+\sqrt{5})},\sqrt{-1})$;
\item $\{[-1], [10]\}$: $\mathbb{Q}_2(\sqrt{1+\sqrt{10}},\sqrt{-1})$, $\mathbb{Q}_2(\sqrt{3+\sqrt{10}},\sqrt{-1})$;
\item $\{[-2], [2]\}$: $\mathbb{Q}_2(\sqrt{\sqrt{2}},\sqrt{-2})$, $\mathbb{Q}_2(\sqrt{3\sqrt{2}},\sqrt{-2})$; 
\item $\{[-5], [5]\}$: $\mathbb{Q}_2(\sqrt{4+\sqrt{5}},\sqrt{-5})$, $\mathbb{Q}_2(\sqrt{2(4+\sqrt{5})},\sqrt{-5})$; 
\item $\{[-2], [-10]\}$: $\mathbb{Q}_2(\sqrt{-2+\sqrt{-2}},\sqrt{-10})$, $\mathbb{Q}_2(\sqrt{-6+\sqrt{-2}},\sqrt{-10})$;
\item $\{[-10], [10]\}$: $\mathbb{Q}_2(\sqrt{\sqrt{10}},\sqrt{-10})$, $\mathbb{Q}_2(\sqrt{3\sqrt{10}},\sqrt{-10})$;
\item $\{[-5], [-10]\}$: $\mathbb{Q}_2(\sqrt{1+\sqrt{-10}},\sqrt{-5})$, $\mathbb{Q}_2(\sqrt{5+\sqrt{-10}},\sqrt{-5})$;
\item $\{[-2], [-5]\}$: $\mathbb{Q}_2(\sqrt{1+\sqrt{-2}},\sqrt{-5})$, $\mathbb{Q}_2(\sqrt{5+\sqrt{-2}},\sqrt{-5})$. 
\end{itemize}

\subsection{A list of $\U_4(\F_2)$-extensions of $\Q_2$} 
The number of admissible pairs $([b],V)$ is $4$. We also have $V$ is uniquely determined  by $[b]$, and $([b],V)$ is admissible if and only if $[b]$ is in $\{[-1],[-2],[-5],[-10]\}$.
Each admissible pair $([b],V)$ can be extended to four admissible triples $([b], V, W)$. Recall that $K$ is the maximal abelian 2-elementary extension $\Q_2(\sqrt{-1},\sqrt{2},\sqrt{5})$ of $\Q_2$.
Here is a list of Galois $\U_4(\F_2)$-extensions of $\Q_2$:
\begin{itemize}
\item $[b]=[-1]$: 
\[
\begin{aligned}
L_1 &= K(\sqrt{1+\sqrt{2}},\sqrt{3+\sqrt{10}},\sqrt{4+\sqrt{2}+\sqrt{10}}),\\
L_2 &= K(\sqrt{1+\sqrt{2}},\sqrt{\dfrac{1+\sqrt{10}}{3}}, \sqrt{\dfrac{4+3\sqrt{2}+\sqrt{10}}{3}}),\\
L_3 &= K(\sqrt{\dfrac{3+\sqrt{2}}{\sqrt{-7}}},\sqrt{3+\sqrt{10}},\sqrt{\dfrac{3+\sqrt{2}}{\sqrt{-7}}+3+\sqrt{10}}),\\
L_4 &= K(\sqrt{\dfrac{3+\sqrt{2}}{\sqrt{-7}}},\sqrt{\dfrac{1+\sqrt{10}}{3}}, \sqrt{\dfrac{3+\sqrt{2}}{\sqrt{-7}}+\dfrac{1+\sqrt{10}}{3}}),
\end{aligned}
\]
where $\sqrt{-7}=1+2^2+2^4+2^5+\cdots\in \Q_2$.
\item $[b]=[-2]$: 
\[
\begin{aligned}
L_5 &= K(\sqrt{\sqrt{2}},\sqrt{\sqrt{\frac{-2}{14}}(2+\sqrt{-10})},\sqrt{\sqrt{2}+ \sqrt{\frac{-2}{14}}(2+\sqrt{-10})}),\\
L_6 &= K(\sqrt{\sqrt{2}},\sqrt{\sqrt{\frac{-2}{94}}(2+3\sqrt{-10})},\sqrt{\sqrt{2}+\sqrt{\frac{-2}{94}}(2+3\sqrt{-10})}),\\
L_7 &= K(\sqrt{\sqrt{\frac{-2}{14}}(4+\sqrt{2})},\sqrt{\sqrt{\frac{-2}{14}}(2+\sqrt{-10})},\sqrt{\sqrt{\frac{-2}{14}}(4+\sqrt{2})+\sqrt{\frac{-2}{14}}(2+\sqrt{-10})}),\\
L_8 &= K(\sqrt{\sqrt{\frac{-2}{14}}(4+\sqrt{2})},\sqrt{\sqrt{\frac{-2}{94}}(2+3\sqrt{-10})},\sqrt{\sqrt{\frac{-2}{14}}(4+\sqrt{2}) +\sqrt{\frac{-2}{94}}(2+3\sqrt{-10})}),\\
\end{aligned}
\]
where 
\[
\begin{aligned}
\sqrt{-2/14} &=1+2^2+2^3+2^4+2^7+\cdots\in \Q_2, \\
\sqrt{-2/94} &=1+2^3+2^4+2^5+2^6+\cdots\in \Q_2.
\end{aligned}
\]
\item $[b]=[-5]$:
\[
\begin{aligned}
L_9 &= K(\sqrt{\sqrt{\frac{-5}{3}}(1+\sqrt{-2})},\sqrt{\sqrt{\frac{-5}{11}}(-1+\sqrt{-10})},\sqrt{\sqrt{\frac{-5}{3}}(1+\sqrt{-2})+\sqrt{\frac{-5}{11}}(-1+\sqrt{-10})}),\\
L_{10} &= K(\sqrt{\sqrt{\frac{-5}{3}}(1+\sqrt{-2})},\sqrt{\sqrt{\frac{-5}{35}}(5+\sqrt{-10})},\sqrt{\sqrt{\frac{-5}{3}}(1+\sqrt{-2})+\sqrt{\frac{-5}{35}}(5+\sqrt{-10}) }),\\
L_{11} &= K(\sqrt{\sqrt{\frac{-5}{3}}(-1+\sqrt{-2})}, \sqrt{\sqrt{\frac{-5}{11}}(-1+\sqrt{-10})}, \sqrt{\sqrt{\frac{-5}{3}}(-1+\sqrt{-2})+\sqrt{\frac{-5}{11}}(-1+\sqrt{-10})}),\\
L_{12} &= K(\sqrt{\sqrt{\frac{-5}{3}}(-1+\sqrt{-2})}), \sqrt{\sqrt{\frac{-5}{35}}(5+\sqrt{-10})},\sqrt{\sqrt{\frac{-5}{3}}(-1+\sqrt{-2})+\sqrt{\frac{-5}{35}}(5+\sqrt{-10})}),
\end{aligned}
\]
where 
\[
\begin{aligned}
\sqrt{\frac{-5}{3}} &=1+2+2^4+2^5+26+2^7+2^9+\cdots\in \Q_2,\\
\sqrt{\frac{-5}{11}} &= 1+2^3+2^6+2^7+2^{10}+\cdots\in \Q_2,\\
\sqrt{\frac{-5}{35}} &= 1+2+2^5+2^6+2^{9}+\cdots\in \Q_2.
\end{aligned}
\]

\item $[b]=[-10]$:
\[
\begin{aligned}
L_{13}  &= K(\sqrt{\sqrt{\frac{-10}{38}}(6+\sqrt{-2})},\sqrt{\sqrt{\frac{-10}{6}}(-1+\sqrt{-5})},\sqrt{\sqrt{\frac{-10}{38}}(6+\sqrt{-2})+ \sqrt{\frac{-10}{6}}(-1+\sqrt{-5})}),\\
L_{14}  &= K(\sqrt{\sqrt{\frac{-10}{38}}(6+\sqrt{-2})},\sqrt{\sqrt{\frac{-10}{70}}(5+3\sqrt{-5})},\sqrt{\sqrt{\frac{-10}{38}}(6+\sqrt{-2})+ \sqrt{\frac{-10}{70}}(5+3\sqrt{-5})}),\\
L_{15}  &= K(\sqrt{\sqrt{\frac{-10}{6}}(2+\sqrt{-2})},\sqrt{\sqrt{\frac{-10}{6}}(-1+\sqrt{-5})},\sqrt{\sqrt{\frac{-10}{6}}(2+\sqrt{-2})+ \sqrt{\frac{-10}{6}}(-1+\sqrt{-5})}),\\
L_{16}  &= K(\sqrt{\sqrt{\frac{-10}{6}}(2+\sqrt{-2})},\sqrt{\sqrt{\frac{-10}{70}}(5+3\sqrt{-5})},\sqrt{\sqrt{\frac{-10}{6}}(2+\sqrt{-2})+ \sqrt{\frac{-10}{70}}(5+3\sqrt{-5})}),
\end{aligned}
\]
where
\[
\begin{aligned}
\sqrt{\frac{-10}{38}}&=1+2 + 2^3+ 2^7 +2^8 +2^9+\cdots \in \Q_2,\\
\sqrt{\frac{-10}{6}}&=1+2 + 2^4+ 2^5 +2^6 +2^7+\cdots \in \Q_2,\\
\sqrt{\frac{-10}{70}}&=1+2 + 2^5+ 2^6 +2^9 +2^12+\cdots \in \Q_2.
\end{aligned}
\]
\end{itemize}

\end{document}